\documentclass[12pt]{amsart}
\usepackage{latexsym}
\usepackage{amsmath}
\usepackage{amsfonts}
\usepackage{amsthm}

\usepackage{tabto}
\usepackage{ifthen}
\usepackage{marginnote}
\usepackage{fancyhdr}
\usepackage{xcolor}
\usepackage{mdframed}
\usepackage[marginparwidth=2.4cm]{geometry}

\usepackage{amssymb}
\usepackage{ifthen}
\usepackage{enumerate}
\usepackage[T1]{fontenc}
\usepackage{dutchcal}
\usepackage{cite}
\usepackage{graphicx}

\usepackage[mathscr]{eucal}
\newcommand*{\rom}[1]{\expandafter\@slowromancap\romannumeral #1@}

\DeclareFontFamily{U}{mathx}{}
\DeclareFontShape{U}{mathx}{m}{n}{<-> mathx10}{}
\DeclareSymbolFont{mathx}{U}{mathx}{m}{n}
\DeclareMathAccent{\widehat}{0}{mathx}{"70}
\DeclareMathAccent{\widecheck}{0}{mathx}{"71}
\def\eq#1{{\rm(\ref{E#1})}}
\def\Eq#1#2{\ifthenelse{\equal{#1}{*}}
  {\begin{equation*}\begin{aligned}#2\end{aligned}\end{equation*}}
  {\begin{equation}\begin{aligned}\label{E#1}#2\end{aligned}\end{equation}}}

\newcounter{allenv}[section]

\newtheorem{thm}{Theorem}
\newtheorem*{thm*}{Theorem}
\newtheorem{prop}{Proposition}
\newtheorem*{prop*}{Proposition}
\newtheorem*{conj*}{Conjecture}
\newtheorem{coro}{Corollary}

\newtheorem*{lemm*}{Lemma}

\theoremstyle{remark}

\newtheorem*{exmp*}{Example}

\newtheorem*{opn*}{Open problem}

\theoremstyle{definition}

\newcommand{\id}{\mbox\small{\rm -id}}

\def\bt{\mathrm{BT}}
\def\bts{\mathrm{BT}^\#}

\author[T. Kiss]{Tibor Kiss}

\title[Strictly nonexpansive, strictly monotone quasi graph-additive functions]{Strictly nonexpansive, strictly monotone quasi graph-additive functions}

\address{Institute of Mathematics,
University of Debrecen,
4002 Debrecen, Pf.~400, Hungary}
\email{kiss.tibor@science.unideb.hu}



\thanks{The research was supported by the Tempus Public Foundation and HUN-REN Hungarian Research Network.}

\def\eq#1{{\rm(\ref{E#1})}}

\def\Eq#1#2{\ifthenelse{\equal{#1}{*}}
	{\begin{equation*}\begin{aligned}[]#2\end{aligned}\end{equation*}}
	{\begin{equation}\begin{aligned}\label{E#1}#2\end{aligned}\end{equation}}}

\begin{document}

\begin{abstract}
In this paper, we provide a negative answer to an open problem concerning the functional equation
\Eq{*}{f(f(-x)+x)=f(-f(x))+f(x),}
namely by showing that the family of continuous solutions is too rich to admit a complete description. Instead, we characterize the solutions within a certain subfamily.
\end{abstract}

\maketitle

\section{Introduction}

In the 2025 paper \cite{Mat25}, Janusz Matkowski proved that if a function $M:\mathbb{R}\times\mathbb{R}\to\mathbb{R}$ is \emph{translative} and \emph{weakly associative} (for further details on this topic, see also \cite[Głazowska--Matkowski]{GlaMat25}), that is, if for all $x,y,u\in\mathbb{R}$, we have
\Eq{*}{
M(x+u,y+u)=M(x,y)+u
\qquad\text{and}\qquad
M(M(x,y),x)=M(x,M(y,x)),
}
respectively, then the single variable function $f(x)=M(x,0)$, $x\in\mathbb{R}$ must satisfy
\Eq{0}{
f(f(-x)+x)=f(-f(x))+f(x),\qquad x\in\mathbb{R}.}
In \cite{Mat25}, the author did not solve the above functional equation, but posed the determination of continuous solutions as an open problem, and formulated a conjecture regarding the solutions. 

To fully understand the motivation behind the conjecture, we must take a brief detour and it is essential to mention the work of W. Jarczyk \cite{Jar88,Jar91}, who solved the co-equation
\Eq{jar}{f(f(x)+x)=f(f(x))+f(x),\qquad x\in\mathbb{R}}
much earlier under remarkably weak assumptions. To be more precise, he verified that all continuous solutions of equation \eq{jar}, that is, continuous \emph{graph-additive function}s, are either linear with a negative slope or positively homogeneous on both the negative and the positive half-line with positive slopes not necessarily equal. Although equation \eq{jar} looks innocent and the solutions are elementary functions, the proof is highly intricate and requires rather sophisticated techniques. Background to Jarczyk’s work, as well as further related results, can be found in the papers \cite[Zdun]{Zdu72}, \cite[Forti]{For83,For84}, \cite[Sablik]{Sab85} and \cite[Matkowski]{Mat85,Mat93}.

Motivated by the striking similarity between the above equations and by Jarczyk's result, Matkowski conjectured that \eq{0} and \eq{jar} have identical continuous solutions. The fact that this conjecture is false follows directly from the results of paper \cite{Mat25}. 

Indeed, the quasi-sum $M(x,y)=\ln(\exp(x)+\exp(y))$, $x,y\in\mathbb{R}$ is associative (hence it is also weakly associative) and, due to the choice of the generator, it is also translative. Consequently the function $x\mapsto M(x,0)=\ln(\exp(x)+1)$ is a continuous solution of \eq{jar}, which is obviously not of Jarczyk type.

Following this, the question arises whether one obtains solutions similar to those of Jarczyk by assuming the function to be positively homogeneous on the non-positive half-line, for instance. The answer is in the negative. This line of inquiry was settled in \cite{Kis26}.

Among other results, we demonstrate in the present paper that Problem 1 formulated in \cite{Mat25} cannot be solved in its full generality. Then we present a subclass of continuous functions where the solutions of \eq{0} can be characterized in some sense.

\section{Notation and conventions}
Let $I$ stand for one of the intervals $\mathbb{R}_-:=(-\infty,0]$, $\mathbb{R}_+:=[0,+\infty)$ or $\mathbb{R}$. Then, for a function $\varphi:D\subseteq\mathbb{R}\to\mathbb{R}$ we shall write $\varphi\in\bt(I)$ if $I\subseteq D$ and both of the inequalities
\Eq{btc}{
x\varphi(x)\geq 0
\qquad\text{and}\qquad
x(x-\varphi(x))\geq 0
}
are satisfied for all $x\in I$. We note that \eq{btc} provides that $\varphi(I^\circ)\subseteq I$ holds. The subfamily of $\bt(I)$ which consists of those functions for which both of the inequalities in \eq{btc} are strict whenever $x\neq 0$ will be denoted by $\bts(I)$. The notation is motivated by the observation that, for $I=\mathbb{R}$, condition \eq{btc} forces the graph of the function into a \emph{bow-tie-shaped} region. 

We emphasize that, if $\varphi\in\bt(I)$, then we have no information on the value $\varphi(0)$. We only know that $\varphi(0-)=0$ or $\varphi(0+)=0$, provided that the limit in question can be interpreted.

It is easy to see that $\bt(I)$ and $\bts(I)$ endowed with the composition of functions form a monoid and a semi-group, respectively. If $S\in\{\bt(I),\bts(I)\}$ and $\varphi_1,\varphi_2\in S$, then we denote the commutator $[\varphi_1,\varphi_2]:=\varphi_1\circ \varphi_2-\varphi_2\circ \varphi_1$ in the usual way. The sole purpose of this notation is to allow, in certain cases, a more concise yet clear formulation of the statements and proofs.

The conjugate $\alpha\circ \varphi\circ\alpha^{-1}:\alpha(I)\to\mathbb{R}$ of $\varphi\in S$ by an injective function $\alpha:I\to\mathbb{R}$ is referred to shortly as $\varphi^\alpha$. 

We shall often refer to the restriction of functions to the non-positive or non-negative half-line. Therefore, if $\varphi:D\to\mathbb{R}$, then we will write $\varphi_-:=\varphi|_{\mathbb{R}_-}$ and $\varphi_+:=\varphi|_{\mathbb{R}_+}$. (It should be noted that in the standard literature these symbols usually refer to the positive and negative parts of a function. Since no more suitable notation seems available, we shall nevertheless employ them here to denote the restriction of the function to the non-positive and non-negative half-line.)

In some of our results, we are going to study those $\mathbb{R}\to\mathbb{R}$ solutions of \eq{0} whose restriction to the non-positive half-line equals a prescribed function. It is easy to verify that $f$ is a solution if and only if $f^{\id}$ is as well. In light of this, our condition is not restrictive: solutions prescribed on the non-negative half-line are simply the $(-\mathrm{id})$-conjugates of the ones we deal with. We do not state these dual results separately and leave their reformulation to the reader.

The above special conjugate will continue to play an important role throughout. One readily checks that $\varphi\in\bt(I)$ (or $\varphi\in\bts(I)$) if and only if $\varphi^{\id}\in\bt(-I)$ (or $\varphi^{\id}\in\bts(-I)$). Since the following relations are elementary, we omit their proofs but will refer to them later:
\Eq{pmi}{
(\varphi^{\id})^{\id}=\varphi,
\qquad
(\varphi_{\mp})^{\id}=(\varphi^{\id})_{\pm},
\qquad\text{and}\qquad
(\mathrm{id}-\varphi)^{\id}=\mathrm{id}-\varphi^{\id}.
}

\section{Solutions depending on arbitrary functions}

In this section, we show that any function satisfying certain weak conditions on the non-positive half-line can be extended to a solution. The next lemma is a straightforward generalization of Proposition 1 of \cite{Kis26}.

\begin{lemm*}
A function $f\in\bt(\mathbb{R})$ is a solution of equation \eq{0} if and only if
\Eq{sol}{
[\mathrm{id}-(f_-)^{\id}, f_+]=0
\qquad\text{and}\qquad
[(f_-)^{\id}, \mathrm{id}-f_+]=0}
hold on the non-negative half-line.
\end{lemm*}

\begin{proof}
Assume that $f$ solves equation \eq{0}. If $x\geq 0$, then by \eq{btc}, we have $f(x)\geq 0$ and $f(-x)+x\geq 0$. Therefore the left and right hand side of equation \eq{0} can be written as
\Eq{*}{
f(f(-x)+x)=f_+(x-(f_-)^{\id}(x))
\quad\text{and}\quad
f(-f(x))+f(x)=f_+(x)-(f_-)^{\id}(f_+(x)),
}
respectively. This shows that the first part of \eq{sol} holds true. Conversely, if the first identity of \eq{sol} is valid, reversing the above argument, we get that $f$ satisfies \eq{0} for $x\geq 0$.

To complete the proof, it is enough to use the fact that $f$ satisfies \eq{0} for $x\leq 0$ if and only if $f^{\id}$ satisfies \eq{0} for $x\geq 0$. But this, in view of the previous part of the proof, is equivalent to $[\mathrm{id}-((f^{\id})_-)^{\id}, (f^{\id})_+]=0$. By \eq{pmi}, we have
\Eq{*}{
\mathrm{id}-((f^{\id})_-)^{\id}
=\mathrm{id}-((f_+)^{\id})^{\id}
=\mathrm{id}-f_+
\qquad\text{and}\qquad
(f^{\id})_+=(f_-)^{\id},
}
which completes the argument.
\end{proof}

The above result has a striking consequence for the abundance of solutions to \eq{0}. This demonstrates that a complete description of all solutions (even continuous ones) to our equation is impossible.

\begin{coro}
Let $\varphi\in\bt^\#(\mathbb{R}_-)$ be any function. Then $f:\mathbb{R}\to\mathbb{R}$ defined by
\Eq{*}{
f(x)=\varphi_-(x)\quad\text{if}\quad x\leq 0
\qquad\text{and}\qquad
f(x)=x-(\varphi_-)^{\id}(x)\quad\text{if}\quad 0\leq x}
solves functional equation \eq{0}.
\end{coro}

\begin{proof}
By the definition of $f$, we have
\Eq{*}{
\mathrm{id}-(f_-)^{\id}=\mathrm{id}-(\varphi_-)^{\id}=f_+
\quad\text{and}\quad
\mathrm{id}-f_+=\mathrm{id}-(\mathrm{id}-(\varphi_-)^{\id})=(f_-)^{\id}.}
Since any function of $\mathrm{BT}(\mathbb{R}_+)$ commutes with itself over the non-negative half-line, \eq{sol} of the Lemma is satisfied. Consequently, $f$ solves \eq{0}.
\end{proof}

\section{Strictly nonexpansive, strictly monotone solutions}

Although the family of solutions is rather rich, the corresponding solutions are nevertheless well-characterizable for certain functions. For the sake of simplicity, from here on, denote $X:=\mathbb{R}_+\setminus\{0\}$.

\begin{prop}\label{Comm}
Let $g,h:X\to X$ be such that $g$ is continuous and strictly increasing. Then $[g,h]=0$ if and only if there exist a constant $\omega>0$, a decreasing homeomorphism $\alpha:X\to\mathbb{R}$ and a function $P:\mathbb{R}\to\mathbb{R}$ periodic with $\omega$, such that
\Eq{mo}{
g(x)=\alpha^{-1}(\alpha(x)+\omega)
\qquad\text{and}\qquad
h(x)=\alpha^{-1}(P(\alpha(x))+\alpha(x)),\qquad x>0.}
\end{prop}

\begin{proof}
Assume that $[g,h]=0$ holds. Due to our conditions on $g$, by Theorem 2.3.8 of the book \cite{KucChoGer90}, for given $\omega>0$, there exists a homeomorphism $\alpha:X\to\mathbb{R}$ such that $\alpha(g(x))=\alpha(x)+\omega$ holds for any $x>0$. Hence the first part of \eq{mo} follows. Then it is also easy to see that $[g^\alpha, h^\alpha]=0$, which is equivalent to writing that $h^\alpha(u+\omega)=h^\alpha(u)+\omega$ holds for all $u\in\mathbb{R}$. Define $P:\mathbb{R}\to\mathbb{R}$ by $P(u):=h^\alpha(u)-u$. Then
\Eq{*}{
P(u+\omega)
=h^\alpha(u+\omega)-(u+\omega)
=h^\alpha(u)+\omega-(u+\omega)
=P(u),\qquad u\in\mathbb{R}.
}
Expressing $h$ in $h^\alpha=P+\mathrm{id}$, we obtain the second part of assertion \eq{mo}.

As for the converse, let $x>0$ be any. Then, by \eq{mo}, we have that
\Eq{*}{
g(h(x))=\alpha^{-1}(P(\alpha(x))+\alpha(x)+\omega)=\alpha^{-1}(P(\alpha(x)+\omega)+\alpha(x)+\omega)=h(g(x)),}
which shows that $g$ and $h$ commute.
\end{proof}

Note that the function $\alpha$ is referred to as the \emph{Abel function} of $g$.

In what follows, we wish to apply Proposition \ref{Comm} to the functions appearing in the Lemma. In order to do this, first we introduce the following concept. We say that a function $\varphi:J\subseteq\mathbb{R}\to\mathbb{R}$ is \emph{strictly nonexpansive} (or \emph{strictly $1$-Lipschitz}) if, for all $x,y\in J$ with $x\neq y$, we have $|\varphi(x)-\varphi(y)|<|x-y|$. 

Obviously, strictly nonexpansive functions are always continuous. It is also easy to verify that a strictly increasing function $\varphi$ is strictly nonexpansive if and only if its displacement map $\mathrm{id}-\varphi$ is strictly increasing.

\begin{thm}
Let $f\in\bts(\mathrm{\mathbb{R}})$ be strictly increasing and strictly nonexpansive. Then $f$ solves functional equation \eq{0} if and only if there exist constants $\rho,\omega>0$, decreasing homeomorphisms $\alpha,\beta:X\to\mathbb{R}$ and positive functions $P:\mathbb{R}\to\mathbb{R}$ and $Q:\mathbb{R}\to\mathbb{R}$ periodic with $\rho$ and $\omega$, respectively, such that
\Eq{A}{
f(x)
=-\alpha^{-1}(\alpha(-x)+\rho)
=x+\beta^{-1}(\beta(-x)+\omega),\qquad x<0
}
and
\Eq{B}{
f(x)
=x-\alpha^{-1}(P(\alpha(x))+\alpha(x))
=\beta^{-1}(Q(\beta(x))+\beta(x)),
\qquad x>0.
}
\end{thm}

\begin{proof}
There is no need to prove the statement in both directions. Instead, we will establish mutually equivalent statements.

In view of our Lemma, $f$ is a solution to the equation \eq{0} if and only if
\Eq{comm}{
[\mathrm{id}-(f_-)^{\id}, f_+]=0
\qquad\text{and}\qquad
[(f_-)^{\id}, \mathrm{id}-f_+]=0}
are valid over the non-negative half-line. Since $f$ is continuous, $f(0)=0$ follows, hence it is enough to ensure that \eq{comm} is valid on the positive half-line. However, the latter -- upon applying Proposition \ref{Comm} with the choices $g:=(f_-)^{\id}|_X$, $h:=\big(\mathrm{id}-f_+\big)|_X$ and $g:=\big(\mathrm{id}-(f_-)^{\id}\big)|_X$, $h:=f_+|_X$, respectively -- is equivalent to \eq{A} and \eq{B} for some constants $\rho,\omega>0$, some decreasing homeomorphisms $\alpha,\beta:X\to\mathbb{R}$, and functions $P:\mathbb{R}\to\mathbb{R}$ periodic with $\rho$ and $Q:\mathbb{R}\to\mathbb{R}$ periodic with $\omega$.

To show that $P$ is positive, let $u\in\mathbb{R}$ be any. Since $f(\alpha^{-1}(u))>0$, then, in view of the first representation of $f$ in \eq{B}, we obtain that $\alpha^{-1}(P(u)+u)<\alpha^{-1}(u)$. Using that $\alpha$ is strictly decreasing, it follows that $P(u)>0$. The assertion for $Q$ can be proved in a similar manner.
\end{proof}

Evidently, the preceding theorem is just a reformulation of our Lemma whenever this can be done. The ultimate objective would be to establish the connection between the functions $\alpha$, $\beta$ and $P$, $Q$ involved in the representations. Furthermore, we note that in the above theorem, it is sufficient to assume that $f$ is strictly nonexpansive and strictly increasing only on the non-positive or non-negative half-line.






\section{Generated solutions}

In the last part of the paper, we address a phenomenon that was already encountered in \cite{Kis26}, though it was not explicitly discussed there. Suppose we seek the solution to \eq{0} by prescribing its form on the non-positive half-line. We shall refer this part as the \emph{generator} of the solutions. According to Proposition 1 of \cite{Kis26}, if the generator is positively homogeneous and the function $f$ satisfies the strict bow-tie condition, it is a solution if and only if $f_+$ is homogeneous with respect to $a$ and $1-a$, where $a$ is the slope of the prescribed branch. Naturally, any positively homogeneous function satisfies this latter condition, and it turns out to be the only possibility when the ratio $\frac{\ln a}{\ln(1-a)}$ is irrational. If $\frac{\ln a}{\ln(1-a)}$ is rational, however, additional exotic extensions become possible (see Figure \ref{fig:elso_kep} below).

\begin{figure}[htbp]
    \centering
    \includegraphics[width=0.7\textwidth]{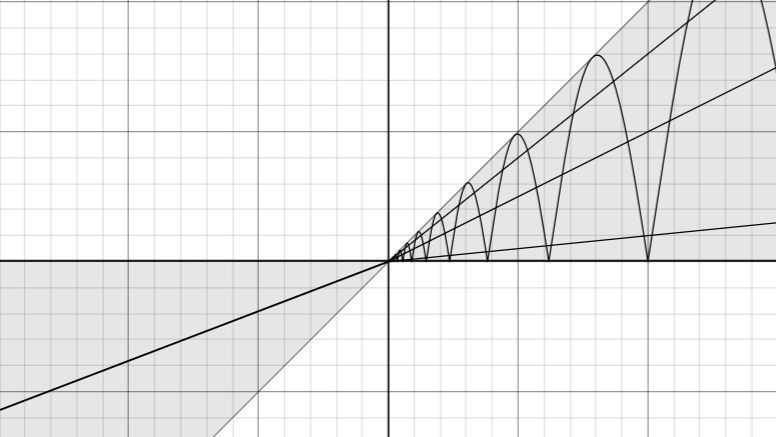} 
    
    \caption{\small If the slope of the left-hand generator is $a_0=\frac{3-\sqrt{5}}{2}$, that is, the ratio $\frac{\ln a_0}{\ln(1-a_0)}$ equals $2$, then on the right-hand side, further branches emerge beyond the 'half-lines'. The formula for the non-positively homogeneous branch is $x\mapsto x\cdot\big|\sin\big(\frac{2\pi}{\ln a_0}\ln x\big)\big|$ for $x>0$.}
    
    \label{fig:elso_kep}
\end{figure}


Obviously, in that special case, we obtain the same family of solutions if the slope of the generator is $1-a$, or equivalently, if we replace the original generator with its displacement map.

For a given function $\gamma\in\bt(\mathbb{R}_-)$, we say that $f\in\bt(\mathbb{R})$ belongs to the family $\mathrm{Sol}(\gamma)$ if and only if $f_-=\gamma_-$ and $f$ solves equation \eq{0}. Furthermore $\mathrm{Sol}^\#(\gamma):=\mathrm{Sol}(\gamma)\cap\bts(\mathbb{R})$. We note that in this latter case we naturally assume that $\gamma\in\bts(\mathbb{R})$. In both cases, the function $\gamma$ is called the \emph{generator} of the family of solutions in question.

In the remaining part, we investigate the general relationship between the solution families generated by a given function $\gamma$ and its displacement map $\mathrm{id}-\gamma$. The main theorem of the section characterizes the case when these two families coincide.

\begin{prop}\label{SolIdentity}
Let $\gamma\in\bt(\mathbb{R}_-)$. Then, for the function classes just introduced, we have
\Eq{*}{
\mathrm{id}-\mathrm{Sol}(\gamma)=\mathrm{Sol}(\mathrm{id}-\gamma)
\qquad\text{and}\qquad
\mathrm{id}-\mathrm{Sol}^\#(\gamma)=\mathrm{Sol}^\#(\mathrm{id}-\gamma).
}
\end{prop}

\begin{proof}
By the Lemma, we have $f\in\mathrm{Sol}(\gamma)$ if and only if
\Eq{*}{
[\mathrm{id}-(\gamma_-)^{\id}, f_+]=0
\qquad\text{and}\qquad
[(\gamma_-)^{\id}, \mathrm{id}-f_+]=0.
}
Applying the third identity in \eq{pmi}, the above system then is equivalent to
\Eq{*}{
[(\mathrm{id}-\gamma_-)^{\id}, f_+]=0
\qquad\text{and}\qquad
[\mathrm{id}-(\mathrm{id}-\gamma_-)^{\id}, \mathrm{id}-f_+]=0.} 
By the Lemma again, these together are equivalent to saying that $\mathrm{id}-f\in\mathrm{Sol}(\mathrm{id}-\gamma)$.

Obviously, $f\in\bts(\mathbb{R})$ if and only if $\mathrm{id}-f\in\bts(\mathbb{R})$, thus the second identity in our assertion holds as well.
\end{proof}

In the proof of Thorem 2 we shall heavily rely on the Theorem quoted below, which we formulated based on what is described in the paper \cite{Sab90} by M. Sablik. For related results see \cite[Dhombres]{Dho70}, \cite[Zdun]{Zdu72}, \cite[Forti]{For83}, \cite[Sablik]{Sab85} or \cite[Paneah]{Pan07}.

\begin{thm*}
Let $r_1, r_2:X\to X$ and $F:X\to\mathbb{R}$ be functions such that $r_1$ and $r_2$ are continuous with $0<r_1, r_2<\mathrm{id}$ and $r_1+r_2=\mathrm{id}$, and that the limit $\lim_{x\to 0^+}\frac{F(x)}{x}$ exists and finite. Then
\Eq{sab}{
F(x)=F(r_1(x))+F(r_2(x)),\qquad x>0
}
is satisfied if and only if there exists a constant $a\in\mathbb{R}$ such that $F(x)=ax$ for $x>0$.
\end{thm*}

\begin{thm}
For $\gamma\in\bts(\mathbb{R}_-)$, the following statements are equivalent.
\begin{enumerate}[{\rm (i)}]\itemsep=1mm
\item\label{i} There exists $0<a<1$ such that $\gamma(x)=ax$ for $x\leq 0$.
\item\label{ii} The family $\mathrm{Sol}^\#(\gamma)$ contains a continuous function, the limit $\lim_{x\to 0^-}\frac{\gamma(x)}{x}$ exists and finite, and
\Eq{same}{\mathrm{Sol}(\gamma)=\mathrm{Sol}(\mathrm{id}-\gamma).}
\end{enumerate}
\end{thm}

\begin{proof}
If we assume (\ref{i}), then $\lim_{x\to0^-}\frac{\gamma(x)}{x}=a$. Define $f:\mathbb{R}\to\mathbb{R}$ by $f(x):=ax$. Then Proposition 1 of paper \cite{Kis26} yields $f\in\mathrm{Sol}^\#(\gamma)$ and \eq{same}.

Assume (\ref{ii}), let $f\in\mathrm{Sol}^\#(\gamma)$ be continuous and, for the sake of simplicity, denote $F:=(\gamma_-)^{\id}$. Then, by the Lemma and our assumption \eq{same}, it follows that $[\mathrm{id}-F, f_+]=0$ and $[\mathrm{id}-F, \mathrm{id}-f_+]=0$. Consequently, for $x\geq 0$, we obtain
\Eq{*}{
F(x)+f(x-F(x))&=f(x)+F(x-f(x)),\\
f(x-F(x))&=f(x)-F(f(x)).
}
Subtracting the second equation from the first, it follows that
\Eq{*}{
F(x)=F(x-f(x))+F(f(x)),\qquad x>0.
}

Applying Theorem under $r_1:=\mathrm{id}-f_+$ and $r_2:=f_+$, we obtain that there exists $a\in\mathbb{R}$ such that
$F(x)=(\gamma_-)^{\id}(x)=ax$ for $x>0$. This yields that $\gamma(x)=ax$ if $x<0$, where, since $\gamma\in\bts(\mathbb{R}_-)$, we must have $0<a<1$. Finally, we also have that $\gamma(0)=\gamma_-(0)=f_-(0)=f(0)=0$, where the last equality uses that $f\in\bts(\mathbb{R}_-)$ is continuous.
\end{proof}

\end{document}